\documentclass[11pt]{article}
\usepackage{mathrsfs}
\usepackage{amsfonts}
\usepackage{amssymb,amsmath,amsthm}
\newcommand{\R}{\mathbb{R}}

\newcommand{\beq}{\begin{equation}}
\newcommand{\eeq}{\end{equation}}
\newcommand{\bp}{\begin{proof}}
\newcommand{\ep}{\end{proof}}
\newcommand{\bo}{\begin{proposition}}
\newcommand{\eo}{\end{proposition}}
\newcommand{\bl}{\begin{lemma}}
\newcommand{\el}{\end{lemma}}

\newtheorem{theorem}{Theorem}[section]

\newtheorem{lemma}[theorem]{Lemma}
\newtheorem{proposition}{Proposition}

\theoremstyle{definition}

\newtheorem{remark}{Remark}

\begin{document}
\title{\LARGE\bf{Maximum principles for nonlinear integro-differential equations and symmetry of solutions}$\thanks{{\small This work was partially supported by NSFC(12031012), NSFC(11831003) and Shanghai Jiao Tong University Scientific and Technological Innovation Funds.}}$ }
\date{}
 \author{Huxiao Luo$^{1,2}$$\thanks{{\small Corresponding author. E-mail: luohuxiao@zjnu.edu.cn (H. Luo), xmq157@sjtu.edu.cn (M. Xu).}}$, Meiqing Xu$^{2}$\\
\small 1 Department of Mathematics, Zhejiang Normal University, Jinhua, Zhejiang, 321004, P. R. China\\
\small 2 School of Mathematical Sciences, CMA-Shanghai, Shanghai Jiao Tong University, P. R. China
}
\maketitle
\begin{center}
\begin{minipage}{13cm}
\par
\small  {\bf Abstract:} In this paper, we study the semilinear integro-differential equations
\begin{equation*}
\mathcal{L}_{K}u(x)\equiv C_n\text{P.V.}\int_{\R^n}\left(u(x)-u(y)\right)K(x-y)dy=f(x,u),
\end{equation*}
and the full nonlinear integro-differential equations
\begin{equation*}
F_{G,K}u(x)\equiv C_n\text{P.V.}\int_{\R^n}G(u(x)-u(y))K(x-y)dy=f(x,u),
\end{equation*}
where $K(\cdot)$ is a symmetric jumping kernel and $K(\cdot)\geq C|\cdot|^{-n-\alpha}$,
$G(\cdot)$ is some nonlinear function without non-degenerate condition. We adopt the direct method of moving planes to study the symmetry and monotonicity of solutions for the integro-differential equations, and investigate the limit of some non-local operators $\mathcal{L}_{K}$ as $\alpha\to2.$ Our results extended some results obtained in \cite{CL} and \cite{CLLG}.


\end{minipage}
\end{center}

 {\section{Introduction}}
 \setcounter{equation}{0}
In the first part of this article, we study the integro-differential equations with linear non-local operator
\begin{equation}\label{20231012-e0}
\mathcal{L}_{K}u(x)\equiv C_{n}\text{P.V.}\int_{\R^n}(u(x)-u(x+y))K(y)dy=f(x,u),\quad \text{in}~\R^n,
\end{equation}
where $P.V.$ denotes the Cauchy principal value integral, the kernel $K$ is a positive function with the properties that $K(-y) =K(y)$ and
\begin{itemize}
\item[($K_1$)] $\forall y\in\R^n\setminus\{0\}$, $K(y)\geq(2-\alpha)\frac{c}{|y|^{n+\alpha}}$ for some $c>0$,
where $\alpha\in(0,2)$.
\end{itemize}
 In order $\mathcal{L}_Ku(x)$ to make sense in $\R^n$, we require that $u\in C^{1,1}_{loc}(\R^n)\cap L^\infty(\R^n)$ and $K$ satisfies the standard L\'{e}vy-Khintchine condition
 \begin{equation}\label{20231103-e100}
\int_{\R^n}\frac{|y|^2}{|y|^2+1}K(y)dy<+\infty,
\end{equation}
see \cite{CS1}.
The operator $\mathcal{L}_{K}$ arises in stochastic control problems with purely jump L\'{e}vy processes, see \cite{CS1,CRS,SYK}.
As a model for $K$, we can take the function
$$K(y)=\frac{1}{|y|^{n+\alpha}}~~\forall y\in\R^n\setminus\{0\},~~0<\alpha<2.$$
In this case, up to some normalization constant, $\mathcal{L}_{K}$ is the fractional Laplacian $(-\Delta)^{\frac{\alpha}{2}}$. The fractional Laplacian is a nonlocal operator, which makes the existence, symmetry, monotonicity and regularity of solutions for fractional Laplacian equation difficult to study. To circumvent the non-locality of the fractional Laplacian, L. Caffarelli and L. Silvestre \cite{CS} introduced an extension method that provides a local realization of the fractional Laplacian by means of a divergence operator in the upper half-space $\R^{n+1}_+$. A series of fruitful results about fractional Laplacian equation have been obtained by the extension method, see e.g., \cite{BCPS,CXSY,FLS,OS} and the references therein. Other approaches of studying fractional Laplacian equation rely on available Green function representations associated with $(-\Delta)^{\frac{\alpha}{2}}$, see, e.g., \cite{CFY,CLO,CZSR,FW,LiZhuo}. However, either by the extension method or by the Green function representation, some extra conditions on the solutions need to be assumed. In \cite{CLL}, W. Chen, C. Li and Y. Li applied a direct method of moving planes for the fractional Laplcaian, and obtained symmetry, monotonicity, and non-existence of the positive solutions. This method has been extensively explored in prior works such as \cite{berestycki1988monotonicity,caffarelli,chen1997priori,chenw,li} and further developed in recent research contributions, including \cite{chenw,CLL,CLO,cheng2017direct}.

In this work, by establishing maximum principle for anti-symmetric functions, decay
at infinity and narrow region principle for $\mathcal{L}_K$, we then use the direct method of moving planes \cite{CLL} to prove the symmetry and monotonicity of the positive solutions for \eqref{20231012-e0}. The nonlocal operators $\mathcal{L}_K$ include the fractional Laplacian but also more general operators which may be anisotropic and may
have varying order. Additionally, we also consider nonlocal operator $\mathcal{L}_{\mathcal{K}}$
$$\mathcal{L}_{\mathcal{K}}u(x)=C_n\text{P.V.}\int_{\R^n}(u(x)-u(x+y))\mathcal{K}(y)dy,$$
with the exponential decay kernel
\begin{equation}\label{20231105-e1}\mathcal{K}(y)=\frac{1}{\Gamma(\frac{2-\alpha}{2})}\frac{e^{-|y|^2}}{|y|^{n+\alpha}},\end{equation}
where $\Gamma(\cdot)$ is the Gamma function. This kind of operator with exponential decay kernel was introduced by
L. Caffarelli and L. Silvestre \cite{CS1}.

In order to get the symmetry of solutions, we assume that $K(y)$ is monotonically decreasing with respect to $|y_i|$ ($i=1,\cdot\cdot\cdot,n$) where $i$ is the i-th component of $y$, i.e., ($K_2$) or ($K'_2$)
\begin{itemize}
\item[($K_2$)] for any $y'\in\R^{n-1}$, $y_i, \bar{y}_i\in \R$ with $y_i^2<\bar{y}_i^2$ we have $K(y_i,y')> K(\bar{y}_i,y');$
\item[($K'_2$)] for any $y=(y_i,y')$, there is a function $\bar{K}_i\in C^1(\R^n\setminus\{0\})$ such that $\bar{K}_i(y_i^2,y')=K(y_i,y')$ and
$\partial_i \bar{K}_i<0.$
\end{itemize}
\begin{remark}
The exponential decay kernel $\mathcal{K}$ is a positive even function, and satisfies the standard L\'{e}vy-Khintchine condition
\eqref{20231103-e100} and the monotonically decreasing condition ($K_2$) ($K'_2$). However, condition ($K_1$) doesn't hold for $\mathcal{K}$. So, if the condition ($K_1$) is used in somewhere of this article, we will also prove this part by replacing this special kernel $\mathcal{K}$ with the general function $K$.
\end{remark}

Now we give some kernel functions which satisfy conditions \eqref{20231103-e100}, ($K_1$), ($K_2$) and ($K'_2$).
\begin{itemize}
\item[(i)] The kernel $K$ has the form
$$K(y)=(2-\alpha)\frac{y^{T}\Lambda y}{|y|^{n+2+\alpha}},~~0<\alpha<2,$$
where
$$\Lambda=diag\{\lambda_1,\cdot\cdot\cdot,\lambda_n\},~~0<\lambda_1\leq\lambda_2\leq\cdot\cdot\cdot\leq\lambda_n.$$
\item[(ii)] The kernel functions of fractional Laplacian after matrix transformation
$$K(y)=(2-\alpha)\frac{1}{\det{\Lambda}|\Lambda^{-1}y|^{n+\alpha}}~~0<\alpha<2.$$
\item[(iii)] The operators of order varying between $\alpha$ and $\beta$
$$K(y)=\frac{1}{|y|^{n+\beta}}~\text{for}~ |y|\leq 1~\text{and }~K(y)=\frac{1}{|y|^{n+\alpha}}~\text{for}~ |y|\geq 1,~~0<\alpha\leq\beta<2.$$
\item[(iv)] The anisotropic fractional Laplacian kernel (see \cite{BS,JW,PS})
\begin{equation}\label{20231105-e2}\mathbb{K}(y)=(2-\alpha)\frac{1}{\|y\|^{n+\alpha}},~~0<\alpha<2,\end{equation}
where the norm
$$\|y\|=|y|_p:=\left(\sum_{i=1}^n|y_i|^{p}\right)^{1/p},~~1 \leq p < \infty.$$
By the equivalence of norms in $\R^n$, it is easy to verify that the anisotropic fractional Laplacian kernel satisfies conditions \eqref{20231103-e100}, ($K_1$), ($K_2$) and ($K'_2$).
\end{itemize}

 %where
%$$L_{K}:=\left\{u:\R^n\to\R \big|\int_{\R^n\setminus B_1(0)}|u(x)|K(x)dx<\infty\right\}.$$

Under conditions ($K_1$) and ($K_2$), we study equation \eqref{20231012-e0} in three cases: (i) bounded domain; (ii) whole space; (iii) half space.
\begin{theorem}\label{th1.1} (i)
Assume that $u\in C_{loc}^{1,1}(B_1(0))\cap C(\overline{B_1(0)})$ is a positive solution of
\begin{equation}\label{20231012-e311}
\left\{
\begin{array}{ll}
\aligned
&\mathcal{L}_Ku(x)=f(u(x)),\quad x\in B_1(0), \\
&u(x)\equiv0,\quad x\notin B_1(0).
\endaligned
\end{array}
\right.
\end{equation}
Assume that $f (\cdot)$ is
Lipschitz continuous. Then $u$ must be radially symmetric and monotone decreasing about the origin.

(ii)
Assume that $u\in C_{loc}^{1,1}(\R^n)\cap L^\infty(\R^n)$ is a positive solution of
\begin{equation}\label{20231012-e321}
\mathcal{L}_Ku(x)=g(u(x)),\quad x\in\R^n.
\end{equation}
 Suppose, for some
$\gamma> 0,$
$$u(x) = o\left(\frac{1}{|x|^\gamma}\right),\quad\text{as}~ |x|\to\infty,$$
and
$$g'(s) \leq s^q,\quad\text{with}~ q\gamma \geq\alpha.$$
Then $u$ must be radially symmetric and monotone decreasing about some point in $\R^n$.

(iii)
Assume that $u\in C_{loc}^{1,1}(\R^n_+)\cap L^\infty(\R^n_+)$ is a nonnegative solution of
\begin{equation}\label{20231013-e1}
\left\{
\begin{array}{ll}
\aligned
&\mathcal{L}_Ku(x)=h(u(x)),\quad x\in\R^n_+, \\
&u(x)\equiv0,\quad x\notin \R^n_+,
\endaligned
\end{array}
\right.
\end{equation}
where
$$\R^n_+ = \{x = (x_1, \cdot\cdot\cdot, x_n) | x_n > 0\}.$$
Suppose that $h(s)$ is Lipschitz continuous in the range of $u$, and $h(0) = 0$. If
\begin{equation}\label{20231014-e1}
\liminf\limits_{|x|\to\infty} u(x) = 0,
\end{equation}
then $u \equiv 0.$
\end{theorem}

In the second part of this article, we extend the direct method of moving planes to the generalized fully nonlinear nonlocal operators, which don't satisfy non-degenerate conditions and are more general than the fractional p-Laplacian. Indeed, the direct method of moving planes has been developed to study fully nonlinear fractional equations under non-degenerate condition \cite{CLLG} and fractional p-Laplacian equations \cite{CL}.
In \cite{CLLG}, W. Chen, C. Li and G. Li studied the fully nonlinear non-local equation
 \begin{equation}\label{20230921-e1}
 C_{n,\alpha}\text{P.V.}\int_{\R^n}\frac{G(u(x)-u(z))}{|x-z|^{n+\alpha}}dz=u^q(x),
 \end{equation}
 where $G$ satisfies
 \begin{itemize}
\item[($G_1$)] $G\in C(\R)$ is an odd function, $G(0)=0$, $G$ is strictly monotone increasing for all $t\in\R$,
\end{itemize}
 and the following non-degenerate condition
 \begin{equation}\label{20231104-a1}
G'(t)\geq c_0>0.
  \end{equation}
This non-degenerate condition plays an indispensable role in proving the narrow region principle and decay at infinity, which are the key ingredients for carrying on the method of moving planes.
We point out that the non-degenerate condition is very strict, and we can see that in addition to identity function $G(t)=t$ satisfying this condition, nonlinear functions $G(t)=|t|^{p-2}t$ do not satisfy this condition.

The non-local equation \eqref{20230921-e1} with the case $G(t)=|t|^{p-2}t$ is studied by W. Chen and C. Li in \cite{CL}. The authors established a boundary estimate lemma to overcome the difficulty that the fractional p-Laplacian doesn't satisfy the non-degenerate condition. The boundary estimate is a variant of the Hopf Lemma, and plays the role of the narrow
region principle. By the boundary estimate, the authors proved radial symmetry and
monotonicity for positive solutions in a unit ball and in the whole space. In this proof, the following property of the function $G(t) = |t|^{p-2}t$ is crucial.
\begin{proposition}\label{20231030-p1} For $G(t) = |t|^{p-2}t$, by the mean value theorem, we
have
$$G(t_2) - G(t_1) = G' (\xi)(t_2 - t_1).$$
Then there exists a constant $c_0 > 0$, such that
$$|\xi| \geq c_0 \max\{|t_1|, |t_2|\}. $$
\end{proposition}

We consider the integro-differential equation with nonlinear nonlocal operator
\begin{equation}\label{20230921-e3}
F_{G,K}(u(x))\equiv C_{n,\alpha}P.V.\int_{\R^n}G(u(x)-u(z))K(x-z)dz=f(u(x)),\quad x\in\R^n.
\end{equation}
Without the non-degenerate condition \eqref{20231104-a1}, under conditions ($K_1$) and ($K'_2$) we prove
\begin{theorem}\label{th1.2} (i) Let $f'(t) \leq 0$ for $t$ sufficiently small. Assume that $u\in C_{loc}^{1,1}(\R^n)\cap L^\infty(\R^n)$ is a positive solution of \eqref{20230921-e3} with
\begin{equation}\label{e32:202}
\liminf\limits_{|x|\to\infty} u(x)= 0.
\end{equation}
Then $u$ must be radially symmetric and monotone decreasing about some point in $\R^n$.

(ii) Let $f'(t) > 0$ for $t\in\R^n$, and \begin{itemize}
\item[($G_2$)]
$$\limsup\limits_{t\to0^+}\frac{f'(t)}{G'(t)}<+\infty.$$
\end{itemize}
Assume that $u\in C_{loc}^{1,1}(B_1(0))\cap C(\overline{B_1(0)})$ is a positive solution of
\begin{equation}\label{20230921-e2}
\left\{
\begin{array}{ll}
\aligned
&F_{G,K}(u(x))=f(u(x)),\quad x\in B_1(0), \\
&u(x)\equiv0,\quad x\notin B_1(0),
\endaligned
\end{array}
\right.
\end{equation}
where $q\geq\gamma+1$. Then $u$ must be radially symmetric and monotone decreasing about the origin.

(iii) Let $f'(t) > 0$ for $t\in\R^n$, and \begin{itemize}
\item[($G'_2$)] there exist $C_1, C_2>0$ and $\varepsilon>0$ such that
$$\frac{G(t_1)-G(t_2)}{t_1-t_2}\geq C_1t_2^{\gamma},~~\frac{f(t_1)-f(t_2)}{t_1-t_2}\leq C_2t_2^{s}~~\text{with}~\gamma<s,\quad\forall 0<t_1<t_2<\varepsilon.$$
\end{itemize}
Assume that $u\in C_{loc}^{1,1}(\R^n)\cap L^\infty(\R^n)$ is a positive solution of \eqref{20230921-e3}
 and
\begin{equation}\label{20230921-e4}
u(x)\thicksim \frac{1}{|x|^\beta}~\text{for}~|x|~\text{sufficiently~large~and~for~}\beta>\frac{\alpha}{s-\gamma}.
\end{equation}
Then $u$ must be radially symmetric and monotone decreasing about some point in $\R^n$.
\end{theorem}
\begin{remark}
A typical example of $G(\cdot)$ and $f(\cdot)$ which satisfy conditions ($G_2$) and ($G'_2$):
$$G(t) = |t|^{\gamma}t,~~f(t) = |t|^{s}t,\quad \gamma<s.$$
We point out that the crucial property of the function $G(t) = |t|^{p-2}t$ (Proposition \ref{20231030-p1}) isn't used in the proof of Theorem \ref{th1.2}-(ii). And by assuming that $G(\cdot)$ satisfies the mild assumption ($G'_2$), we use Cauchy mean value theorem to overcome this difficulty, see Sec. 3.3.
\end{remark}
Finally, we investigate the limit of $\mathcal{L}_Ku(x)$ as $\alpha\to2$ for each fixed $x$ and discover some
interesting phenomenons.
For
$$K(y)=(2-\alpha)\frac{C_n}{\det{\Lambda}|\Lambda^{-1}y|^{n+\alpha}},$$
from \cite[(6.1)]{CS1} we know that
$$\lim\limits_{\alpha\to2}
\mathcal{L}_Ku(x)=-\sum_{i=1}^n\lambda_i^{2}\partial_{ii} u(x).$$
Indeed, it is well know that
$$(-\Delta)^{\frac{\alpha}{2}}u(x)=C_n(2-\alpha)\text{P.V.}\int_{\R^n}\frac{u(x)-u(y)}{|x-y|^{n+\alpha}}dy\to-\Delta u(x),\quad \text{as}~ \alpha\to2.$$
then by variable substitution
$$C_n(2-\alpha)\text{P.V.}\int_{\R^n}\frac{u(x)-u(y)}{\det{\Lambda}|\Lambda^{-1}(x-y)|^{n+\alpha}}dy\to-\sum_{i=1}^n\lambda_i^{2}\partial_{ii} u(x).$$

Via Taylor's expansion, we prove
\begin{theorem}\label{th1.3} (i) Assume that $u\in C_{loc}^{1,1}(\R^n)\cap L^\infty(\R^n)$. Let $$\mathcal{L}_{\mathcal{K}}u(x)=\frac{4n}{\omega_n}\text{P.V.}\int_{\R^n}(u(x)-u(x+y))\mathcal{K}(y)dy,$$
where $\mathcal{K}$ is the exponential decay kernel defined by \eqref{20231105-e1}.
Then
$$\lim\limits_{\alpha\to2^-}
\mathcal{L}_{\mathcal{K}}u(x)=-\Delta u(x).$$

(ii) Assume that $u\in C_{loc}^{1,1}(\R^n)\cap L^\infty(\R^n)$. Let
$$\mathcal{L}_{\mathbb{K}}u(x)\equiv C_{n}\text{P.V.}\int_{\R^n}(u(x)-u(x+y))\mathbb{K}(y)dy,$$
where $\mathbb{K}$ is the anisotropic fractional Laplacian kernel defined by \eqref{20231105-e2}.
Then
$$\lim\limits_{\alpha\to 2^-}
\mathcal{L}_{\mathbb{K}}u(x)=-C_{n,p}\Delta u(x).$$
\end{theorem}

The paper is organized as follows. In Section 2, we give maximum principle for anti-symmetric functions, decay at infinity and narrow region principle for the linear operator $\mathcal{L}_{K}$, and prove Theorems \ref{th1.1}. In Sections 3, we give maximum principle for anti-symmetric functions and a boundary estimate for the nonlinear operator $F_{G,K}$, and prove Theorems \ref{th1.2}.
In Sections 4, we study the limit of $\mathcal{L}_Ku(x)$ as $\alpha\to2$.

Throughout the paper, we use $C$ to denote positive constants whose values may vary from line to line.

\vskip4mm
{\section{Nonlinear equations $\mathcal{L}_Ku(x)=f(x,u)$}}
 \setcounter{equation}{0}
In this section, we always assume that ($K_1$) and ($K_2$) hold. We give the maximum principle for anti-symmetric functions, decay at infinity and narrow region principle for the nonlocal operator $\mathcal{L}_K$, and then use them to prove Theorem \ref{th1.1} by the direct method of moving planes.
\vskip4mm
{\subsection{Maximum principle for anti-symmetric functions, decay at infinity and narrow region principle for $\mathcal{L}_K$}}

For any real number $\lambda$, let
$$T_\lambda = \{x\in\R^n | x_1 = \lambda\}$$
be a plane perpendicular to $x_1-$axis. Let $\Sigma_\lambda$ be the region to the left of the plane $T_\lambda$
$$\Sigma_\lambda = \{x\in\R^n |x_1 < \lambda\},$$
and
$$x^\lambda = (2\lambda-x_1, x_2,\cdot\cdot\cdot,x_n)$$
be the reflection of the point $x = (x_1, x_2,\cdot\cdot\cdot,x_n)$ about the plane $T_\lambda$.
Denote $u_\lambda(x)=u(x^\lambda)$, $w_\lambda(x)=u(x^\lambda)-u(x)$, and for simplicity of notation, we also denote $w_\lambda$ by $w$ and $\Sigma_\lambda$ by $\Sigma$.

\begin{theorem}\label{t2:2023113} (Maximum principle for anti-symmetric functions) Let $\Omega$ be a bounded
domain in $\Sigma$. Assume that $w\in C^{1,1}_{loc}(\Omega)$ and is lower semi-continuous on $\bar{\Omega}$. If
\begin{equation}\nonumber
\left\{
\begin{array}{ll}
\aligned
&\mathcal{L}_Kw(x)\geq 0,~~&x\in\Omega, \\
&w(x)\geq 0,~~&x\in\Sigma\setminus\Omega,\\
&w(x^\lambda)=-w(x),~~&x\in\Sigma,
\endaligned
\end{array}
\right.
\end{equation}
then $w\geq 0$ in $\Omega$. Moreover, if $w(x) = 0$ for some point inside $\Omega$, then $w\equiv 0$ almost everywhere in $\R^n$. The same conclusions holds for unbounded domains $\Omega$ if we further assume that
$$\liminf\limits_{|x|\to\infty}w(x)\geq0.$$
\end{theorem}
\begin{proof}
Suppose otherwise, then there exists a point $x^o\in\Omega$ such that
$$w(x^o)=\min\limits_{\Omega}w=\min\limits_{\Sigma}w<0.$$
By dividing $\R^n$ into the sum of $\Sigma$ and $\Sigma^c$, and using integral variable substitution for the integral on $\Sigma^c$, we have
\begin{equation}\label{2023113-e11}
\aligned
&\mathcal{L}_Kw(x^o)\\
= &P.V.\int_{\Sigma}\left[w(x^o)-w(y)\right]\left[K(x^o-y)-K(x^o-y^\lambda)\right]dy +\int_{\Sigma}2w(x^o)K(x^o-y^\lambda)dy\\
=:& I_1+I_2.
\endaligned
\end{equation}
Since $K(z)$ is monotonically decreasing with respect to $|z_1|$ and
\begin{equation*}
\aligned
\left|(x^o-y)_1\right|<\left|(x^o-y^\lambda)_1\right|~\text{for}~x^o, y\in\Sigma,
\endaligned
\end{equation*}
 we can infer that
\begin{equation*}
\aligned
I_1\leq 0.
\endaligned
\end{equation*}
By the positivity of $K(\cdot)$, we have
\begin{equation*}
\aligned
I_2<0.
\endaligned
\end{equation*}
Hence $\mathcal{L}_Kw(x^o)<0$. This contradicts our assumption, hence
\begin{equation}\label{e11:20231103}
w(x)\geq0,\quad\forall x\in\Sigma.
\end{equation}
It follows that if $w(x^o) = 0$ at some point $x\in\Omega$, then $u_\lambda(x^o)=u(x^o)$, hence
\eqref{2023113-e11} holds with $I_2 = 0$. Now, our assumption implies $I_1 \geq 0$. Consequently
$$w(y)\leq 0,\quad\text{almost~everywhere~in~}\Sigma.$$
Combining this with \eqref{e11:20231103},
$$w(y) = 0,\quad \text{almost~everywhere~in~} \Sigma,$$
and from the antisymmetry of $w$, we arrive at
$$w(y) = 0,\quad\text{almost~everywhere~in~} \R^n.$$
\end{proof}

\begin{theorem}\label{t3:2023113} (Decay at infinity) Let $\Omega$ be an unbounded region in $\Sigma$. Let $w\in
C^{1,1}_{loc}(\Omega)\cap L^\infty(\Omega)$ be a solution of
\begin{equation}\label{20231012-e10}
\left\{
\begin{array}{ll}
\aligned
&\mathcal{L}_Kw(x)+c(x)w(x)\geq0,\quad &x\in \Omega, \\
&w(x)\geq0,\quad &x\in\Sigma\setminus\Omega,\\
&w(x^\lambda)=-w(x),\quad &x\in\Sigma
\endaligned
\end{array}
\right.
\end{equation}
with
\begin{equation}\nonumber
\liminf\limits_{|x|\to\infty}|x|^\alpha c(x)\geq0,
\end{equation}
then there exists a constant $R_0 > 0$ ( depending on $c(x)$, but independent of $w$), such that if
$$w(x^0) = \min\limits_{\Omega}w(x) < 0,$$
then
$|x^0| \leq R_0.$
\end{theorem}
\begin{proof}
Suppose otherwise, then there exists a point $x^o\in\Omega$ such that
$$w(x^o)=\min\limits_{\Omega}w=\min\limits_{\Sigma}w<0.$$
By condition ($K_2$) and \begin{equation*}
\aligned
\left|(x^o-y)_1\right|<\left|(x^o-y^\lambda)_1\right|~\text{for}~x^o, y\in\Sigma,
\endaligned
\end{equation*}
we have
$$K(x^o-y)\geq K(x^o-y^\lambda).$$
Then by $w(y^\lambda)=-w(y)$, we have
\begin{equation}\label{20231012-e112}
\aligned
&\mathcal{L}_Kw(x^o)\\
=&C_{n}\text{P.V.}\int_{\R^n}(w(x^o)-w(y))K(x^o-y)dy \\
=&C_{n}\text{P.V.}\int_{\Sigma}\left[(w(x^o)-w(y))K(x^o-y)+(w(x^o)+w(y))K(x^o-y^\lambda)\right]dy \\
\leq & 2C_{n}w(x^o)\int_{\Sigma}K(x^o-y^\lambda)dy.
\endaligned
\end{equation}

{\bf The general kernel case.} By condition ($K_1$), we have
$$\int_{\Sigma}K(x^o-y^\lambda)dy\geq\int_{\Sigma}\frac{a}{|x^o-y^\lambda|^{n+\alpha}}dy,$$
which together with \eqref{20231012-e112} implies that
\begin{equation*}
\aligned
\mathcal{L}_Kw(x^o)+c(x^o)w(x^o)
\leq  \left[2C_n a\int_{\Sigma}\frac{1}{|x^o-y^\lambda|^{n+\alpha}}dy+c(x^o)\right]w(x^o).
\endaligned
\end{equation*}
Therefore, following the proof of Theorem 2.4 in \cite{CLL}, we have
\begin{equation}\nonumber
\aligned
\mathcal{L}_Kw(x^o)+c(x^o)w(x^o)
\leq  \left[C\frac{1}{|x^o|^\alpha}+c(x^o)\right]w(x^o).
\endaligned
\end{equation}
Then from the equation \eqref{20231012-e10}, we obtain
$$C+|x^o|^\alpha c(x^o)\leq0.$$
Now if $|x^o|$ is sufficiently large, this would contradict the decay assumption on $c(x)$.

{\bf The exponential decay kernel case:}
%$$\mathcal{K}(y)=\frac{1}{\Gamma(\frac{2-\alpha}{2})}\frac{e^{-|y|^2}}{|y|^{n+\alpha}}.$$
Let $\Sigma^c=\R^n\setminus\Sigma$, $x^1=(3|x^o|+x^o_1, (x^o)')$, then
$B_{|x^o|}(x^1)\subset\Sigma^c$ and
$$|x^o-y|\leq 4|x^o|,\quad\forall y\in B_{|x^o|}(x^1).$$
Then
\begin{equation*}
\aligned
\int_{\Sigma}\mathcal{K}(x^o-y^\lambda)dy =&\int_{\Sigma^c}\mathcal{K}(x^o-y)dy\\
=&\frac{1}{\Gamma(\frac{2-\alpha}{2})}\int_{\Sigma^c}\frac{e^{-|x^o-y|^2}}{|x^o-y|^{n+\alpha}}dy \\
\geq &\frac{1}{\Gamma(\frac{2-\alpha}{2})}\int_{B_{|x^o|}(x^1)}\frac{e^{-|x^o-y|^2}}{|x^o-y|^{n+\alpha}}dy\\
\geq &\frac{1}{\Gamma(\frac{2-\alpha}{2})}\int_{B_{|x^o|}(x^1)}\frac{e^{-16|x^o|^2}}{4^{n+\alpha}|x^o|^{n+\alpha}}dy\\
= &C\frac{e^{-16|x^o|^2}}{|x^o|^\alpha}.
\endaligned
\end{equation*}
Therefore
\begin{equation*}
\aligned
\mathcal{L}_Kw(x^o)+c(x^o)w(x^o)
\leq  \left[C\frac{e^{-16|x^o|^2}}{|x^o|^\alpha}+c(x^o)\right]w(x^o).
\endaligned
\end{equation*}
Then from the equation \eqref{20231012-e10}, we obtain
$$C+e^{16|x^o|^2}|x^o|^\alpha c(x^o)\leq0.$$
Now if $|x^o|$ is sufficiently large, this would contradict the decay assumption on $c(x)$.

The proof of the theorem is complete.
\end{proof}

\begin{theorem}\label{t4:20230921} (Narrow region principle) Let $\Omega$ be a bounded narrow region in $\Sigma$, and
$$\Omega\subset\{x| \lambda-\delta < x_1 < \lambda \}$$
with small $\delta$. Suppose that $w\in
C^{1,1}_{loc}(\Omega)\cap L^\infty(\Omega)$ and is lower semi-continuous on $\bar{\Omega}$. If $c(x)$ is
bounded from below in $\Omega$ and
\begin{equation}\label{20231012-e100}
\left\{
\begin{array}{ll}
\aligned
&\mathcal{L}_Kw(x)+c(x)w(x)\geq0,\quad &x\in \Omega, \\
&w(x)\geq0,\quad &x\in\Sigma\setminus\Omega,\\
&w(x^\lambda)=-w(x),\quad &x\in\Sigma,
\endaligned
\end{array}
\right.
\end{equation}
then for sufficiently small $\delta$, we have
$$w(x) \geq 0,\quad x\in \Omega.$$
Furthermore, if $w = 0$ at some point in $\Omega$, then
$$w(x) = 0~\text{a.e.~in~} \R^n.$$
These conclusions hold for unbounded region $\Omega$ if we further assume that
$$\liminf\limits_{|x|\to\infty}
w(x) \geq 0.$$
\end{theorem}
\begin{proof}
Suppose otherwise, then there exists a point $x^o\in\Omega$ such that
$$w(x^o)=\min\limits_{\Omega}w=\min\limits_{\Sigma}w<0.$$
Similar to \eqref{20231012-e112},
\begin{equation}\nonumber
\aligned
\mathcal{L}_Kw(x^o)+c(x^o)w(x^o)\leq \left[2C_{n}\int_{\Sigma}K(x^o-y^\lambda)dy+c(x^o)\right]w(x^o).
\endaligned
\end{equation}

{\bf The general kernel case.}
By ($K_1$) and using the same argument of $(25)$ in \cite[Theorem 2.3]{CLL} (also see \cite{chen2017direct,cheng2017maximum,berestycki1991method,berestycki1988monotonicity}), we have
$$\int_{\Sigma}K(x^o-y^\lambda)dy\geq\int_{\Sigma}\frac{(2-\alpha)c}{|x^o-y^\lambda|^{n+\alpha}}dy\geq C\frac{1}{\delta^\alpha},$$
and thus
\begin{equation*}
\aligned
\mathcal{L}_Kw(x^o)+c(x^o)w(x^o)
\leq  \left[C\frac{1}{\delta^\alpha}+c(x^o)\right]w(x^o).
\endaligned
\end{equation*}
Then from the equation \eqref{20231012-e100}, we obtain
$$C\frac{1}{\delta^\alpha}+c(x^o)\leq0.$$
Now if $\delta$ is sufficiently small, this would contradict that $c(x)$ is
bounded from below in $\Omega$.

{\bf The exponential decay kernel case:}
%$$\mathcal{K}(y)=\frac{1}{\Gamma(\frac{2-\alpha}{2})}\frac{e^{-|y|^2}}{|y|^{n+\alpha}}.$$
Let
$$D = \{y|\delta < y_1 - x^0_1<1,~|y'-(x^o)'|<1\},$$
then
$D\subset \Sigma^c$.
Letting $s=y_1-x^o_1$ and $\tau=|y'-(x^o)'|$, then
\begin{equation*}
\aligned
\int_{\Sigma}\mathcal{K}(x^o-y^\lambda)dy =&\int_{\Sigma^c}\mathcal{K}(x^o-y)dy\\
=&\frac{1}{\Gamma(\frac{2-\alpha}{2})}\int_{\Sigma^c}\frac{e^{-|x^o-y|^2}}{|x^o-y|^{n+\alpha}}dy \\
\geq &\frac{1}{\Gamma(\frac{2-\alpha}{2})}\int_{D}\frac{e^{-|x^o-y|^2}}{|x^o-y|^{n+\alpha}}dy\\
= &\frac{1}{\Gamma(\frac{2-\alpha}{2})}\int_{\delta}^1\int^1_{0}\frac{\omega_{n-2}\tau^{n-2}e^{-s^2(1+t^2)}}{(s^2+\tau^2)^{\frac{n+\alpha}{2}}}d\tau ds\\
= &\frac{1}{\Gamma(\frac{2-\alpha}{2})}\int_{\delta}^1\int^{\frac{1}{s}}_{0}\frac{\omega_{n-2}(st)^{n-2}e^{-s^2(1+t^2)}}{s^{n+\alpha}(1+t^2)^{\frac{n+\alpha}{2}}}sdt ds,
\endaligned
\end{equation*}
where we use the variable substitution $t=\frac{\tau}{s}.$
Then by the elementary inequality
$$e^{-s^2(1+t^2)}\geq e^{-\frac{s^4}{2}}e^{-\frac{(1+t^2)^2}{2}},$$
we get
\begin{equation*}
\aligned
&\int_{\Sigma}\mathcal{K}(x^o-y^\lambda)dy\\
\geq &\frac{1}{\Gamma(\frac{2-\alpha}{2})}\int_{\delta}^1\int^{\frac{1}{s}}_{0}\frac{\omega_{n-2}(st)^{n-2}e^{-\frac{s^4}{2}}e^{-\frac{(1+t^2)^2}{2}}}{s^{n+\alpha}(1+t^2)^{\frac{n+\alpha}{2}}}sdtds\\
=&C(n,\alpha)\int_{\delta}^1 s^{-1-\alpha}e^{-\frac{s^4}{2}}\left[\int^{\frac{1}{s}}_{0}
\frac{t^{n-2}e^{-\frac{(1+t^2)^2}{2}}}{(1+t^2)^{\frac{n+\alpha}{2}}}dt\right]ds \\
\geq&C(n,\alpha)\int_{\delta}^1 s^{-1-\alpha}e^{-\frac{s^4}{2}}ds\int^{1}_{0}
\frac{t^{n-2}e^{-\frac{(1+t^2)^2}{2}}}{(1+t^2)^{\frac{n+\alpha}{2}}}dt\\
\geq&C'(n,\alpha)e^{\frac{1}{2}}\int_{\delta}^1 s^{-1-\alpha}ds\geq C''(n,\alpha)\frac{1}{\delta^\alpha}.
\endaligned
\end{equation*}
Then by the same argument in the general kernel case, we get a contradiction.

The proof of the theorem is complete.
\end{proof}

\vskip4mm
{\subsection{Symmetry and monotonicity in a unit ball}}

In this subsection, we prove Theorem \ref{th1.1}-(i).
\begin{proof}
Let $$\Omega_\lambda=\Sigma_\lambda\cap B_1(0).$$
By equation \eqref{20231012-e311} we have
\begin{equation*}
\mathcal{L}_Kw_\lambda(x)+c_\lambda(x)w_\lambda(x)=0,\quad\text{where}~c_\lambda(x)=-\frac{ f(u_\lambda(x))-f(u(x))}{u_\lambda(x)-u(x)}.
\end{equation*}

{\bf Step 1.} Choose any ray from the origin as
the positive $x_1$ direction. We show that for $\lambda>-1$ but sufficiently close to $-1$, there holds
\begin{equation}\label{e22:2023113}
w_\lambda(x)\geq0,\quad\forall x\in\Omega_\lambda.
\end{equation}
Indeed, the Lipschitz continuity condition on $f$ guarantees that $c_\lambda(x)$ is uniformly bounded
from below. Then by Theorem \ref{t4:20230921} (Narrow region principle), for
 $\lambda>-1$ and sufficiently close to $-1$, \eqref{e22:2023113} holds
since $\Sigma_\lambda$ is a narrow region for such $\lambda$.

{\bf Step 2.} Step 1 provides a starting point to move the plane $T_\lambda$. Now we move the plane
to the right as long as \eqref{e22:2023113} holds to its limiting position. More precisely, define
$$\lambda_o= \sup\{\lambda\leq 0 | w_\mu(x) \geq 0, x\in \Omega_\mu, \mu\leq\lambda\}.$$
We show that
\begin{equation*}
\lambda_o = 0.
\end{equation*}
Suppose in the contrary, $\lambda_o<0$, then by Theorem \ref{t2:2023113} (Maximum principle for anti-symmetric functions), we have
 \begin{equation*}
 w_{\lambda_o}(x)>0,~~\forall x\in\Sigma_{\lambda_o}.
 \end{equation*}
Thus for any $\delta> 0$,
$$w_{\lambda_o} (x) > c_\delta > 0,\quad\forall x\in \Sigma_{\lambda_o-\delta}.$$
By the continuity of $w_\lambda$ with respect to $\lambda$, there exists $\epsilon > 0$, such that
 \begin{equation}\label{e28:2023113}
w_\lambda(x)\geq 0, \quad \forall x\in \Sigma_{\lambda_o-\delta},\quad \forall\lambda\in [\lambda_o, \lambda_o+ \epsilon).
 \end{equation}
Using Theorem \ref{t4:20230921} (Narrow region principle), we have
$$w_\lambda(x) \geq 0,\quad \forall x\in \Sigma_\lambda\setminus \Sigma_{\lambda_o-\delta}.$$
This together with \eqref{e28:2023113} implies
$$w_\lambda(x) \geq 0, \quad \forall x\in\Sigma_\lambda, \quad \forall \lambda\in[\lambda_o, \lambda_o + \epsilon).$$
This contradicts the definition of $\lambda_o$. Therefore, we must have $\lambda_o = 0.$ It follows that
$$w_0(x) \geq 0,  \quad \forall x \in \Sigma_0.$$
Since we can choose the $x_1-$direction arbitrarily, hence $u$ is radially symmetric about
the origin. The monotonicity is a consequence of the fact that
$$w_\lambda(x) > 0,  \quad \forall x \in \Sigma_\lambda,\quad\forall \lambda\in(-1,0].$$
 This completes the proof of Theorem \ref{th1.1}-(i).
\end{proof}

\vskip4mm
{\subsection{Symmetry and monotonicity in $\R^n$ }}

In this subsection, we prove Theorem \ref{th1.1}-(ii).
\begin{proof} By equation \eqref{20231012-e321} and mean value theorem,
%$$g(u_\lambda(x))-g(u(x))=g'(\psi_\lambda(x))w_\lambda(x),$$
\begin{equation*}
\mathcal{L}_Kw_\lambda(x)+c_\lambda(x)w_\lambda(x)=0, \quad\text{where}~c_\lambda(x)=-g'(\psi_\lambda(x)),
\end{equation*}
where $\psi_\lambda(x)$ is between $u_\lambda(x)$ and $u(x)$.

{\bf Step 1}. We show that for $\lambda$ sufficiently negative,
\begin{equation}\label{e32:2023113}
w_\lambda(x)\geq0,\quad\forall x\in\Sigma_\lambda.
\end{equation}
Suppose \eqref{e32:2023113} is violated, then there exists an $x^o\in\Sigma_\lambda$, such that
$$w_\lambda(x^o) = \min\limits_{\Sigma_\lambda}w_\lambda< 0.$$
Since $\lambda$ is sufficiently negative, thus $|x^o|$ is sufficiently large.
Since
$u_\lambda(x^o)< u(x^o),$
we have
$$0 \leq u_\lambda(x^o)\leq\psi_\lambda(x^o)\leq u(x^o).$$
The decay assumptions of $u(x)$ and $g'$ imply that
$$|x^o|^\alpha c_\lambda(x^o)\geq0,\quad\text{for}~|x^o|~\text{sufficiently~large}.$$
However, by the same argument of Theorem \ref{t3:2023113} (Decay at infinity), there exists $R_0 > 0$, such that, if $x^o$ is a negative minimum of $w_\lambda$ in $\Sigma_\lambda$, then
$$|x^o| \leq R_0.$$
This contradicts that $|x^o|$ is sufficiently large. Hence \eqref{e32:2023113} must hold.

{\bf Step 2.}
\eqref{e32:2023113} provides a starting point, from which we move the plane $T_\lambda$
toward the right as long as \eqref{e32:2023113} holds to its limiting position to prove that $u$ is symmetric
about the limiting plane. More precisely, let
$$\lambda_o = \sup\{\lambda | w_\mu(x) \geq 0, x\in\Sigma_\mu, \mu\leq\lambda\},$$
we show that $u$ is symmetric about the limiting plane $T_{\lambda_o}$, i.e.
 \begin{equation}\label{e36:2023113}
w_{\lambda_o}(x)\equiv0,~~\forall x\in\Sigma_{\lambda_o}.
\end{equation}
Suppose \eqref{e36:2023113} is false, then by Theorem \ref{t2:2023113} (Maximum principle for anti-symmetric functions),
 \begin{equation*}
 w_{\lambda_o}(x)>0,~~\forall x\in\Sigma_{\lambda_o}.
 \end{equation*}
It follows that for any positive number $\rho$,
$$w_{\lambda_o}(x)\geq c_0 > 0,\quad\forall x \in \overline{\Sigma_{\lambda_o-\rho}\cap B_{R_0}(0)},$$
where $R_0$ is defined in Step 1. Since $w_\lambda$ depends on $\lambda$ continuously, there exists
a $\delta_0 > 0$ such that for all $\delta\in(0, \delta_0)$,
 \begin{equation*}
 w_{\lambda_o+\delta}(x)\geq  0,\quad\forall x \in \overline{\Sigma_{\lambda_o-\rho}\cap B_{R_0}(0)},
  \end{equation*}
  Now, we can show that
   \begin{equation}\label{e39:2023113}
 w_{\lambda_o+\delta}(x)\geq  0, \quad\forall x \in \Sigma_{\lambda_o+\delta}.
  \end{equation}
Suppose \eqref{e39:2023113} is false, then there exists $x^o\in \Sigma_{\lambda_o+\delta}$, such that
$$ w_{\lambda_o+\delta}(x^o) = \min\limits_{\Sigma_{\lambda_o+\delta}}w_{\lambda_o+\delta} < 0.$$
By Theorem \ref{t3:2023113} (Decay at infinity), there must hold
\begin{equation}\label{e40:2023113}
x^o \in (\Sigma_{\lambda_o+\delta}\setminus \Sigma_{\lambda_o-\rho} ) \cap B_{R_0}(0).
\end{equation}
Since $(\Sigma_{\lambda_o+\delta}\setminus \Sigma_{\lambda_o-\rho} ) \cap B_{R_0}(0)$ is a narrow region for sufficiently small $\delta$ and $\rho$,
and by Theorem \ref{t4:20230921} (Narrow region principle), $w_{\lambda_o+\delta}$ cannot attain its negative minimum
here, which contradicts \eqref{e40:2023113}.
 Hence \eqref{e39:2023113} holds. Thus the plane $T_{\lambda_o}$ can still be moved further to the right, which contradicts with the definition of $\lambda_o$. Therefore, \eqref{e36:2023113} hold.

 Since $x_1$ direction can be chosen arbitrarily, we conclude that $u$ is radially symmetric
about some point.
This completes Theorem \ref{th1.1}-(ii).
\end{proof}

\vskip4mm
{\subsection{Non-existence of solutions on a half space }}
In this subsection, we prove Theorem \ref{th1.1}-(iii).
\begin{proof}
First, for $u(x)\geq0$ we show that
$$\text{either}~ u(x) > 0 ~~\text{or}~~ u(x) \equiv 0, \quad x\in \R^n_+.$$
 Indeed, if there exists $\tilde{x}\in\R^n$ such that $u(\tilde{x})=0$.
 %$u(x) \not\equiv 0$ and $u(x)\geq0$,
%then by \eqref{20231014-e1} and $u\in C_{loc}^{1,1}$,
By equation \eqref{20231013-e1}, we have
$$\int_{\R^n}(u(\tilde{x})-u(y))K(\tilde{x}-y)dy=h(u(\tilde{x}))=h(0)=0.$$
On the other hand, by the strict positivity of $K(\cdot)$,
$$\int_{\R^n}(u(\tilde{x})-u(y))K(\tilde{x}-y)dy=-\int_{\R^n}u(y)K(\tilde{x}-y)dy\leq 0.$$
This implies that
$$u(y)\equiv0,\quad\text{a.e.~on}~\R^n.$$
Hence in the following, we may assume that
$$u(x) > 0, \quad x\in \R^n_+.$$

Now we carry on the method of moving planes on the solution $u$ along $x_n$ direction.
Let
$$T_\lambda = \{x\in \R^n_+ | x_n = \lambda\},\quad \lambda> 0,$$
and
$$\Sigma_\lambda = \{x\in \R^n_+ | 0 < x_n < \lambda\}.$$
Let
$$x^\lambda = (x_1,..., x_{n-1}, 2\lambda- x_n)$$
be the reflection of $x$ about the plane $T_\lambda$. Denote $w_\lambda(x) = u(x^\lambda) - u(x).$
By \eqref{20231013-e1}, we see that $w_\lambda(x)$ satisfies the following equation
\begin{equation*}
\left\{
\begin{array}{ll}
\aligned
&\mathcal{L}_Kw_\lambda(x)+c_\lambda(x)w_\lambda(x)=0,\quad &x\in \R^n_+, \\
%&w(x)\geq0,\quad &x\in\R^n\setminus \Sigma_\lambda,\\
&w_\lambda(x^\lambda)=-w_\lambda(x),\quad &x\in\R^n_+,
\endaligned
\end{array}
\right.
\end{equation*}
where
$$c_\lambda(x)=-\frac{h(u_\lambda(x))-h(u(x))}{u_\lambda(x)-u(x)}$$
is bounded from below since $h(\cdot)$ is Lipschitz continuous.

{\bf Step 1}. For $\lambda$ sufficiently small, since $\Sigma_\lambda$ is a narrow region, by using the same proof of Theorem \ref{t4:20230921} (Narrow region principle), we have
\begin{equation}\label{20231014-e11} w_\lambda(x) \geq 0,\quad \forall x\in\Sigma_\lambda. \end{equation}

{\bf Step 2.}
Let
$$\lambda_o = \sup\{\lambda | w_\mu(x) \geq 0, x\in\Sigma_\mu, \mu\leq\lambda\},$$
We show that
\begin{equation}\label{20231014-e12} \lambda_o = +\infty. \end{equation}
Otherwise, if $\lambda_o < +\infty$, then by \eqref{20231014-e11}, combining Theorem \ref{t4:20230921} (Narrow region principle) and Theorem \ref{t3:2023113} (Decay at infinity) and going through the similar arguments as in the previous subsection, we are able
to show that
$$w_{\lambda_o} (x) \equiv 0\quad\text{in}~ \Sigma_{\lambda_o},$$
which is impossible, since $0=u(x)=u(x^\lambda) > 0$ for $x\in\partial\R^n_+$.

Therefore, \eqref{20231014-e12} holds. Consequently, the solution $u(x)$ is monotone increasing
with respect to $x_n$. This contradicts \eqref{20231014-e1}. So $u \equiv 0.$
\end{proof}

\vskip4mm
{\section{ Full nonlinear equations $F_{G,K} u(x)=f(u(x))$}}
 \setcounter{equation}{0}
In this section, we always suppose that ($K_1$) and ($K'_2$) hold. We give the maximum principle for anti-symmetric functions, and a boundary estimate for the nonlinear nonlocal operator $F_{G,K}$, and then use them to prove Theorem \ref{th1.2} by the direct method of moving planes.

The simple maximum principle for $F_{G,K}$ is not necessary in the proof of Theorem \ref{th1.2}. However, due to its interest in itself, we give the proof in the following, which holds also for $L_{K}$, since only the monotonicity of $G(\cdot)$ is used.
\begin{theorem}\label{t1:20230921} (Simple maximum principle) Let $\Omega\subset\R^n$ be a bounded domain in $\R^n$, and let $u\in C^{1,1}_{loc}(\Omega)$ be a lower-semi-continuous
function in $\bar{\Omega}$ such that
\begin{equation}\label{e1:20230922}
\left\{
\begin{array}{ll}
\aligned
&F_{G,K}(u(x))\geq 0,~~x\in\Omega, \\
&~u \geq 0,\quad x\in\R^n\setminus\Omega.
\endaligned
\end{array}
\right.
\end{equation}
Then
\begin{equation}\label{e2:20230922}
u \geq 0,~~x\in\Omega.
\end{equation}
The same conclusions holds for unbounded domains $\Omega$ if we further assume that
$$\liminf\limits_{|x|\to\infty}u(x)\geq0.$$
\end{theorem}
\begin{proof}
Suppose \eqref{e2:20230922} is violated, then there exists $x^o\in\Omega$
such that
$$u(x^o) = \min\limits_{\Omega} u < 0.$$
Since $u \geq 0$ for all $x\in\R^n\setminus\Omega$, then
$$u(x^o)<u(y),\quad \forall y\in \R^n\setminus\Omega.$$
By the monotonicity of $G(\cdot)$ and $G(0)=0$, we have
$$\int_{\R^n}G(u(x^o)-u(y))K(x^o-y)dy<0.$$
This contradicts \eqref{e1:20230922} and hence proves the theorem.
\end{proof}
\vskip4mm
{\subsection{Maximum principle for anti-symmetric functions and a boundary estimate }}

\begin{theorem}\label{t2:20230921} (Maximum principle for anti-symmetric functions) Let $\Omega$ be a bounded
domain in $\Sigma$. Assume that $w\in C^{1,1}_{loc}(\Omega)$ and is lower semi-continuous on $\bar{\Omega}$. If
\begin{equation}\nonumber
\left\{
\begin{array}{ll}
\aligned
&F_{G,K}u_\lambda(x)-F_{G,K}u(x)\geq 0,~~&x\in\Omega, \\
&w(x)\geq 0,~~&x\in\Sigma\setminus\Omega,\\
&w(x^\lambda)=-w(x),~~&x\in\Sigma,
\endaligned
\end{array}
\right.
\end{equation}
then $w\geq 0$ in $\Omega$. Moreover, if $w(x) = 0$ for some point inside $\Omega$, then $w\equiv 0$ almost everywhere in $\R^n$. The same conclusions hold for unbounded domain $\Omega$ if we further assume that
$$\liminf\limits_{|x|\to\infty}w(x)\geq0.$$
\end{theorem}
\begin{proof}
Suppose otherwise, then there exists a point $x^o\in\Omega$ such that
$$w(x^o)=\min\limits_{\Omega}w=\min\limits_{\Sigma}w<0.$$
By dividing $\R^n$ into the sum of $\Sigma$ and $\Sigma^c$, and using integral variable substitution for the integral on $\Sigma^c$, we have
\begin{equation}\label{2023921-e11}
\aligned
&F_{G,K}u_\lambda(x^o)-F_{G,K}u(x^o)\\
= &C_{n,\alpha}P.V.\int_{\Sigma}\left[G(u_\lambda(x^o)-u_\lambda(y))-G(u(x^o)-u(y))\right]\left[K(x^o-y)-K(x^o-y^\lambda)\right]dy \\
&+C_{n,\alpha}\int_{\Sigma}\left[G(u_\lambda(x^o)-u_\lambda(y))-G(u(x^o)-u_\lambda(y))+G(u_\lambda(x^o)-u(y))-G(u(x^o)-u(y))\right] \\ &\quad\quad\quad\quad\cdot K(x^o-y^\lambda)dy\\
=:& I_1+I_2.
\endaligned
\end{equation}
Since $K(z)$ is monotonically decreasing with respect to $|z_1|$, we have
$$K(x^o-y)-K(x^o-y^\lambda)\geq0,\quad \forall y\in \Sigma.$$
By the monotonicity of $G(\cdot)$, we can infer that
\begin{equation*}
\aligned
I_1\leq 0,\quad\text{by}~\left[u_\lambda(x^o)-u_\lambda(y)\right]-\left[ u(x^o)-u(y)\right]=w(x^o)-w(y)\leq0~\forall y\in\Sigma;
\endaligned
\end{equation*}
and
\begin{equation*}
\aligned
I_2<C_{n,\alpha}\int_{\Sigma}\left[G(u_\lambda(x^o)-u(y))-G(u(x^o)-u(y)) \right] K(x^o-y^\lambda)dy<0,\quad\text{by}~u_\lambda(x^o)-u(x^o)<0.
\endaligned
\end{equation*}
Hence
\begin{equation}\label{e1221:2023922}
F_\alpha(u_\lambda(x^o))-F_\alpha(u(x^o))<0.
\end{equation}
This contradicts our assumption, hence
\begin{equation}\label{e11:2023922}
w(x)\geq0,\quad\forall x\in\Sigma.
\end{equation}
It follows that if $w(x^o) = 0$ at some point $x\in\Omega$, then $u_\lambda(x^o)=u(x^o)$, hence
\eqref{2023921-e11} holds with $I_2 = 0$. Now, our assumption implies $I_1 \geq 0$, consequently
$$G(u_\lambda(x^o)-u_\lambda(y))-G(u(x^o)-u(y))\geq0,$$
and by the monotonicity of $G(\cdot)$, we derive,
$$w(y)\leq 0,\quad\text{almost~everywhere~in~}\Sigma.$$
Combining this with \eqref{e11:2023922},
$$w(y) = 0,\quad \text{almost~everywhere~in~} \Sigma,$$
and from the antisymmetry of $w$, we arrive at
$$w(y) = 0,\quad\text{almost~everywhere~in~} \R^n.$$
\end{proof}

\begin{theorem}\label{t3:20230921} (A boundary estimate). Assume that $w_{\lambda_o}> 0$, for $x\in\Sigma_{\lambda_o}$. Suppose $\lambda_k\searrow\lambda_o$, and $x^k\in\Sigma_{\lambda_k}$, such that
\begin{equation}
w_{\lambda_k}(x^k)=\min\limits_{\Sigma_{\lambda_k}}w_{\lambda_k}\leq0~\text{and}~x^k\to x^o\in\partial\Sigma_{\lambda_o}.
\end{equation}
Let
$\delta_k=dist(x^k,\partial\Sigma_{\lambda_k} )\equiv|\lambda_k-x^k_1|.$
Then
\begin{equation}\label{e1:2023922}
\limsup\limits_{\delta_k\to0}\frac{1}{\delta_k}\left[F_{G,K}(u_{\lambda_k}(x^k))-F_{G,K}(u(x^k))\right]<0.
\end{equation}
\end{theorem}
\begin{proof}
By $u_{\lambda_k}(x^k)\leq u(x^k)$, similar to \eqref{2023921-e11}, we infer
\begin{equation}\label{e12:2023922}
\aligned
&\frac{1}{\delta_k}\left[F_{G,K}(u_{\lambda_k}(x^k))-F_{G,K}(u(x^k))\right]\\
\leq &\frac{C_{n,\alpha}}{\delta_k}P.V.\int_{\Sigma_{\lambda_k}}\left[K(x^k-y)-K(x^k-y^{\lambda_k})\right]\left[G(u_{\lambda_k}(x^k)-u_{\lambda_k}(y))-G(u(x^k)-u(y))\right]dy \\
=:& I_{1k}.
\endaligned
\end{equation}
By ($K'_2$) and mean value theorem, we have
\begin{equation}\label{e13:2023922}
\aligned
&\frac{1}{\delta_k}\left[K(x^k-y)-K(x^k-y^{\lambda_k})\right] \\
=&\frac{1}{\delta_k}\left[\bar{K}_1\left(|x^k_1-y_1|^2,(x^k-y)'\right)-\bar{K}_1\left(|x^k_1-y^{\lambda_k}_1|^2,(x^k-y)'\right)\right]\\
=&\frac{|x^k_1-y_1|^2-|x^k_1-y^{\lambda_k}_1|^2}{\delta_k}\partial_1 \bar{K}_1\left(\eta_k(y),(x^k-y)'\right)\\
=&-4(\lambda_k-y_1)\partial_1 \bar{K}_1\left(\eta_k(y),(x^k-y)'\right)\\
\to&-4(\lambda_o-y_1)\partial_1 \bar{K}_1\left(\eta_o(y),(x^o-y)'\right),
\endaligned
\end{equation}
where
$$|x^k_1-y_1|^2 \leq \eta_k(y) \leq |x^k_1 - y_1^{\lambda_k}|^2$$
and hence
$$|x^o_1 -y_1|^2 \leq \eta_o(y) \leq |x^o_1 - y_1^{\lambda_o}|^2.$$
By condition ($K'_2$), the last term in \eqref{e13:2023922} is positive in $\Sigma_{\lambda_o}$.
Meanwhile, as $k\to\infty$, by the strict monotonicity of $G(\cdot)$,
\begin{equation}\label{e14:2023922}
G(u_{\lambda_k}(x^k)-u_{\lambda_k}(y))-G(u(x^k)-u(y))\to G(u_{\lambda_o}(x^o)-u_{\lambda_o}(y))-G(u(x^o)-u(y)) < 0,
\end{equation}
for all $y\in \Sigma_{\lambda_o}$.
Combining \eqref{e12:2023922}, \eqref{e13:2023922} and \eqref{e14:2023922}, we get \eqref{e1:2023922}.
\end{proof}

\vskip4mm
{\subsection{Symmetry and monotonicity in $\R^n$ in case $f'(t)\leq0$ }}
In this subsection, we consider equation \eqref{20230921-e3} and prove Theorem \ref{th1.2}-(i).

\begin{proof}
{\bf Step 1}. To show that for $\lambda$ sufficiently negative,
\begin{equation}\label{e32:2023114}
w_\lambda(x)\geq0,\quad\forall x\in\Sigma_\lambda.
\end{equation}
Suppose \eqref{e32:2023114} is violated. By \eqref{e32:202}
there exists an $x^o\in\Sigma_\lambda$, such that
$$w_\lambda(x^o) = \min\limits_{\Sigma_\lambda}w_\lambda< 0.$$
And by equation \eqref{20230921-e3},
\begin{equation*}
F_{G,K}u_\lambda(x^o)-F_{G,K}u(x^o)=f(u_\lambda(x^o))-f(u(x^o))=f'(\xi_\lambda(x^o))w_\lambda(x^o),
\end{equation*}
where $$u_\lambda(x^o)\leq \xi_\lambda(x^o) \leq u(x^o).$$
For sufficiently negative $\lambda$, $u(x^o)$ is small, and consequently, $\xi_\lambda(x^o)$ is also small. Then by the condition on $f(\cdot)$, it follows that
 \begin{equation}\label{e35:20231141}
\aligned
F_{G,K}u_\lambda(x^o)-F_{G,K}u(x^o)\geq0.
\endaligned
\end{equation}
While on the other hand, from the proof of \eqref{e1221:2023922} in Theorem \ref{t2:20230921} (Maximum principle for anti-symmetric functions), we have
 \begin{equation*}
\aligned
F_{G,K}u_\lambda(x^o)-F_{G,K}u(x^o)<0.
\endaligned
\end{equation*}
This contradicts \eqref{e35:20231141}. Hence \eqref{e32:2023114} must hold.

{\bf Step 2.}
 \eqref{e32:2023114} provides a starting point, from which we move the plane $T_\lambda$
toward the right as long as \eqref{e32:2023114} holds to its limiting position to prove that $u$ is symmetric
about the limiting plane. More precisely, let
$$\lambda_o = \sup\{\lambda | w_\mu(x) \geq 0, x\in\Sigma_\mu, \mu\leq\lambda\},$$
we show that $u$ is symmetric about the limiting plane $T_{\lambda_o}$, or
 \begin{equation}\label{e36:2023114}
w_{\lambda_o}(x)\equiv0,~~\forall x\in\Sigma_{\lambda_o}.
\end{equation}
Suppose \eqref{e36:2023114} is false, then by Theorem \ref{t2:20230921} (Maximum principle for anti-symmetric functions),
 \begin{equation*}
 w_{\lambda_o}(x)>0,~~\forall x\in\Sigma_{\lambda_o}.
 \end{equation*}
On the other hand, by the definition of $\lambda_o$, there exists a sequence $\lambda_k \searrow \lambda_o$, and
$x^k \in \Sigma_{\lambda_k}$, such that
\begin{equation}\label{e38:2023114}
w_{\lambda_k}(x^k)=\min\limits_{\Sigma_{\lambda_k}}w_{\lambda_k}<0,~~\text{and}~ \nabla w_{\lambda_k}(x^k)=0.
\end{equation}
%We can see from \eqref{e37:20230921}, \eqref{e38:20230921} and the continuity of $w_\lambda(x)$ that
%$$x_k\in\Sigma_{\lambda_k}\setminus\Sigma_{\lambda_o}.$$
Now, we use the assumption about $f$ ($f'(t) \leq0$ for $t$ small)
 to show that the sequence $\{x^k\}$ is bounded. In fact, if $|x^k|$ is sufficiently large, then $u(x^k)$ is small.
Then by equation \eqref{20230921-e3} and \eqref{e38:2023114},
\begin{equation}\label{e39:2023114}
F_{G,K}u_{\lambda_k}(x^k)-F_{G,K}u(x^k)=f(u_{\lambda_k}(x^k))-f(u(x^k))=f'(\xi_{\lambda_k}(x^k))w_\lambda(x^k)\geq 0,
\end{equation}
where $$u_{\lambda_k}(x^k)\leq\xi_{\lambda_k}(x^k)\leq u(x^k).$$
While on the other hand, from the proof of \eqref{e1221:2023922} in Theorem \ref{t2:20230921}, we have
 \begin{equation*}
\aligned
F_{G,K}u_{\lambda_k}(x^k)-F_{G,K}u(x^k)<0.
\endaligned
\end{equation*}
This contradicts \eqref{e39:2023114}. Hence the sequence $\{x^k\}$ must be bounded.

Now from \eqref{e38:2023114}, we have
$$w_{\lambda_o}(x^o)\leq 0,~~\text{hence}~ x^o\in \partial\Sigma_{\lambda_o};~~~~
\text{and}~~ \nabla w_{\lambda_o}(x^o) = 0.$$
It follows that
$$\frac{w_{\lambda_k}(x^k)}{\delta_k}\to0,\quad
\text{as}~ k\to+\infty.$$
Then by \eqref{e39:2023114}, we get
\begin{equation*}
\limsup\limits_{\delta_k\to0}\frac{1}{\delta_k}\left[F_{G,K}(u_{\lambda_k}(x^k))-F_{G,K}(u(x^k))\right]\geq 0.
\end{equation*}
This contradicts Theorem \ref{t3:20230921}. Hence \eqref{e36:2023114} holds. Since $x_1$ direction can be chosen arbitrarily, we conclude that $u$ is radially symmetric
about some point.
This completes the proof of Theorem \ref{th1.2}-(i).
\end{proof}

\vskip4mm
{\subsection{Symmetry and monotonicity in a ball in case $f'(t)>0$}}
In this subsection, we consider equation \eqref{20230921-e2} and prove Theorem \ref{th1.2}-(ii).
\begin{proof}
Let $$\Omega_\lambda=\Sigma_\lambda\cap B_1(0).$$
By equation \eqref{20230921-e2} we have
\begin{equation}\label{e21:20230921}
F_{G,K}u_\lambda(x)-F_{G,K}u(x)=f(u_\lambda(x))-f(u(x)).
\end{equation}
{\bf Step 1.} Choose any ray from the origin as
the positive $x_1$ direction. First we show that for $\lambda>-1$ but sufficiently close to $-1$, we have
\begin{equation}\label{e22:20230921}
w_\lambda(x)\geq0,\quad\forall x\in\Omega_\lambda.
\end{equation}
Suppose otherwise, then there exists a point $x^o\in\Omega_\lambda$, such that
$$w_\lambda(x^o)=\min\limits_{\Omega_\lambda}w_\lambda=\min\limits_{\Sigma_\lambda}w_\lambda<0.$$
With the same argument in the proof of Theorem \ref{t2:20230921}, we have
\begin{equation}\label{e23:20230921}
\aligned
&F_{G,K}u_\lambda(x^o)-F_{G,K}u(x^o)\\
= &C_{n,\alpha}P.V.\int_{\Sigma_\lambda}\left[G(u_\lambda(x^o)-u_\lambda(y))-G(u(x^o)-u(y))\right]\left[K(x^o-y)-K(x^o-y^\lambda)\right]dy \\
&+C_{n,\alpha}\int_{\Sigma_\lambda}\left[G(u_\lambda(x^o)-u_\lambda(y))-G(u(x^o)-u_\lambda(y))+G(u_\lambda(x^o)-u(y))-G(u(x^o)-u(y))\right]\\
&\quad\quad\quad\quad\quad\cdot K(x^o-y^\lambda)dy\\
\leq &C_{n,\alpha}\int_{\Sigma_\lambda} G(u_\lambda(x^o)-u(y))-G(u(x^o)-u(y)) K(x^o-y^\lambda)dy.
\endaligned
\end{equation}
Let $D:=\Sigma_\lambda\setminus\Omega_\lambda$ and by $u(y)=0$ for all $y\in D$, we obtain
\begin{equation}\label{e1:20230928}
F_{G,K}u_\lambda(x^o)-F_{G,K}u(x^o)\leq C_{n,\alpha}\int_{D}\left[G(u_\lambda(x^o))-G(u(x^o))\right]K(x^o-y^\lambda)dy.
\end{equation}
Combining \eqref{e1:20230928} and \eqref{e21:20230921}, we get
\begin{equation*}
f(u_\lambda(x^o))-f(u(x^o))\leq C_{n,\alpha}\int_{D} \left[G(u_\lambda(x^o))-G(u(x^o))\right] K(x^o-y^\lambda)dy.
\end{equation*}
Thus by ($K_1$),
\begin{equation*}
\frac{f(u_\lambda(x^o))-f(u(x^o))}{G(u_\lambda(x^o))-G(u(x^o))}\geq C_{n,\alpha}\int_{D}K(x^o-y^\lambda)dy\geq C'_{n,\alpha}\int_{D}\frac{1}{|x^o-y^\lambda|^{n+\alpha}}dy.
\end{equation*}
Then, by using Cauchy mean value theorem, we derive,
\begin{equation}\label{e4:20230928}
\frac{f'(\xi(x^o))}{G'(\xi(x^o))}\geq C'_{n,\alpha}\int_{D}\frac{1}{|x^o-y^\lambda|^{n+\alpha}}dy\geq C\frac{1}{\delta^\alpha},
\end{equation}
where $$u_\lambda(x^o)\leq \xi(x^o)\leq u(x^o),$$
and $\delta=\lambda+1$ is the width of the region $\Omega_\lambda$ in the $x_1-$direction.
We see from $u\in C(\overline{B_1(0)})$ that for $\lambda$ sufficiently close to $-1$, there exists $\varepsilon>0$ such that
$$0<u_\lambda(x^o)<u(x^o)<\varepsilon.$$
Then by \eqref{e4:20230928} and condition ($G_2$), we get a contradiction. Therefore \eqref{e22:20230921} must be true for $\lambda$ is sufficiently close to $-1$.

{\bf Step 2.} Define
$$\lambda_o= \sup\{\lambda\leq 0 | w_\mu(x) \geq 0, x\in \Omega_\mu, \mu\leq\lambda\}.$$
Now, we show that
\begin{equation}\label{e26:20230921}
\lambda_o = 0.
\end{equation}
Suppose in the contrary, $\lambda_o<0$, then by Theorem \ref{t2:20230921} (Maximum principle for anti-symmetric functions), we have
 \begin{equation*}
 w_{\lambda_o}(x)>0,~~\forall x\in\Omega_{\lambda_o}.
 \end{equation*}
On the other hand, by the definition of $\lambda_o$, there exists a sequence $0\geq\lambda_k \searrow\lambda_o$, and
$x^k \in \Omega_{\lambda_k}$, such that
\begin{equation}\label{e28:20230921}
w_{\lambda_k}(x^k)=\min\limits_{\Sigma_{\lambda_k}}w_{\lambda_k}<0,~~\text{and}~ \nabla w_{\lambda_k}(x^k)=0.
\end{equation}
There is a subsequence of $\{x^k\}$ that converges to some point $x^o$, and
from \eqref{e28:20230921}, we have
$$w_{\lambda_o}(x^o)\leq 0,~~\text{hence}~ x^o\in \partial\Sigma_{\lambda_o};~~~~
\text{and}~~ \nabla w_{\lambda_o}(x^o) = 0.$$
It follows that
$$\frac{w_{\lambda_k}(x^k)}{\delta_k}\to0,\quad
\text{as}~ k\to+\infty.$$
Then by \eqref{e21:20230921}, we get
\begin{equation}\nonumber
\limsup\limits_{\delta_k\to0}\frac{1}{\delta_k}\left[F_{G,K}(u_{\lambda_k}(x^k))-F_{G,K}(u(x^k))\right]\geq 0.
\end{equation}
This contradicts Theorem \ref{t3:20230921}. Hence \eqref{e26:20230921} holds.

Since $x_1$ direction can be chosen arbitrarily, we conclude that $u$ is radially symmetric
about the origin.
This completes the proof of Theorem \ref{th1.2}-(ii).
\end{proof}

\vskip4mm
{\subsection{Symmetry and monotonicity in $\R^n$ in case $f'(t)>0$ }}
In this subsection, we consider equation \eqref{20230921-e3} and prove Theorem \ref{th1.2}-(iii).
\begin{proof}
{\bf Step 1}. To show that for $\lambda$ sufficiently negative,
\begin{equation}\label{e32:20230921}
w_\lambda(x)\geq0,\quad\forall x\in\Sigma_\lambda.
\end{equation}
Suppose \eqref{e32:20230921} is violated, then there exists an $x^o\in\Sigma_\lambda$, such that
$$w_\lambda(x^o) = \min\limits_{\Sigma_\lambda}w_\lambda< 0.$$
And by equation \eqref{20230921-e3},
\begin{equation}\label{e31:20230921}
F_{G,K}u_\lambda(x^o)-F_{G,K}u(x^o)=f(u_\lambda(x^o))-f(u(x^o)).
\end{equation}
Let $R = |x^o|$. Choose a point $x_R \in\Sigma_\lambda$, so that
$$B_R(x_R)\subset\Sigma_\lambda~~\text{and}~~|x_R|=MR.$$
By the decay condition \eqref{20230921-e4}, for any $y\in B_R(x_R)$ we have
\begin{equation*}
u(y)\thicksim\frac{1}{M^\beta R^\beta}, ~~ u(x^o)\thicksim\frac{1}{R^\beta}, \quad\text{for}~R~\text{large}.
\end{equation*}
So, we can choose $M$ sufficiently large such that
\begin{equation}\label{e33:20230921}
u(y)\leq\frac{C_1}{M^\beta R^\beta}\leq \frac{C_2}{R^\beta}\leq u(x^o),~~\forall y\in B_R(x_R).
\end{equation}
Then for $\lambda$ sufficiently negative ($R$ is sufficiently large), there exists a $\varepsilon>0$ such that
\begin{equation*}
0<u(y)\leq u(x^o)<\varepsilon,\quad\forall y\in  B_R(x_R).
\end{equation*}
By \eqref{e23:20230921} and ($K_1$), we have
 \begin{equation}\label{e34:20230921}
\aligned
&F_{G,K}u_\lambda(x^o)-F_{G,K}u(x^o)\\
\leq &C_{n,\alpha}\int_{\Sigma_\lambda}\left[G(u_\lambda(x^o)-u(y))-G(u(x^o)-u(y))\right]K(x^o-y^\lambda)dy\\
\leq &C'_{n,\alpha}\int_{\Sigma_\lambda}\frac{G(u_\lambda(x^o)-u(y))-G(u(x^o)-u(y))}{|x^o-y^\lambda|^{n+\alpha}}dy\\
\leq &C'_{n,\alpha}\int_{B_R(x_R)}\frac{G(u_\lambda(x^o)-u(y))-G(u(x^o)-u(y))}{|x^o-y^\lambda|^{n+\alpha}}dy.
\endaligned
\end{equation}
Note that
$$0<u_\lambda(x^o)<u(x^o)<\varepsilon,\quad\text{and}~u(x^o)-u(y)\thicksim\frac{1}{R^\beta},$$
then by condition ($G'_2$) and \eqref{e34:20230921}, we have
 \begin{equation}\label{e35:20230921}
\aligned
&F_{G,K}u_\lambda(x^o)-F_{G,K}u(x^o)\\
\leq &C_{n,\alpha}c_0w_\lambda(x^o)\int_{B_R(x_R)}\frac{(u(x^o)-u(y))^\gamma}{|x^o-y^\lambda|^{n+\alpha}}dy\\
\leq &C_{n,\alpha}C w_\lambda(x^o)\frac{1}{R^{\beta \gamma+\alpha}}.
\endaligned
\end{equation}
Combining \eqref{e31:20230921} with \eqref{e35:20230921}, by condition ($G'_2$) we get
$$\frac{C}{R^{\beta \gamma+\alpha}}\leq \frac{f(u_\lambda(x^o))-f(u(x^o))}{u_\lambda(x^o)-u(x^o)}\leq C_2u^{s}(x^o).$$
This contradicts assumption \eqref{20230921-e4}. Hence \eqref{e32:20230921} must hold.

{\bf Step 2.}
Step 1 provides a starting point, from which we move the plane $T_\lambda$
toward the right as long as \eqref{e32:20230921} holds to its limiting position. Define
$$\lambda_o = \sup\{\lambda | w_\mu(x) \geq 0, x\in\Sigma_\mu, \mu\leq\lambda\},$$
we show that $u$ is symmetric about the limiting plane $T_{\lambda_o}$, or
 \begin{equation}\label{e36:20230921}
w_{\lambda_o}(x)\equiv0,~~\forall x\in\Sigma_{\lambda_o}.
\end{equation}
Suppose \eqref{e36:20230921} is false, then by Theorem \ref{t2:20230921} (Maximum principle for anti-symmetric functions),
 \begin{equation*}
 w_{\lambda_o}(x)>0,~~\forall x\in\Sigma_{\lambda_o}.
 \end{equation*}
On the other hand, by the definition of $\lambda_o$, there exists a sequence $\lambda_k \searrow \lambda_o$, and
$x^k \in \Sigma_{\lambda_k}$, such that
\begin{equation}\label{e38:20230921}
w_{\lambda_k}(x^k)=\min\limits_{\Sigma_{\lambda_k}}w_{\lambda_k}<0,~~\text{and}~ \nabla w_{\lambda_k}(x^k)=0.
\end{equation}
%We can see from \eqref{e37:20230921}, \eqref{e38:20230921} and the continuity of $w_\lambda(x)$ that
%$$x_k\in\Sigma_{\lambda_k}\setminus\Sigma_{\lambda_o}.$$
Now, we use condition
\eqref{20230921-e4} to show that the sequence $\{x^k\}$ is bounded. In fact, if $|x^k|$ is sufficiently large, then  $$0<u_{\lambda_k}(x^k)<u(x^k)\leq\frac{C_0}{|x^k|^\beta}<\varepsilon.$$
And by equation \eqref{20230921-e3}, \eqref{e38:20230921} and ($G'_2$),
\begin{equation}\label{e39:20230921}
F_{G,K}u_{\lambda_k}(x^k)-F_{G,K}u(x^k)=f(u_{\lambda_k}(x^k))-f(u(x^k))\geq C_2u^{s}(x^k)w_{\lambda_k}(x^k).
\end{equation}

On the other hand, by \eqref{e23:20230921} and ($K_1$), we have
 \begin{equation}\label{e40:20230921}
\aligned
&F_{G,K}u_{\lambda_k}(x^k)-F_{G,K}u(x^k)\\
\leq &C_{n,\alpha}\int_{\Sigma_{\lambda_k}}\frac{G(u_{\lambda_k}(x^k)-u(y))-G(u(x^k)-u(y))}{|x^k-y^{\lambda_k}|^{n+\alpha}}dy\\
\leq &C_{n,\alpha}\int_{ B_{R_k}(x_{R_k})}\frac{G(u_{\lambda_k}(x^k)-u(y))-G(u(x^k)-u(y))}{|x^k-y^{\lambda_k}|^{n+\alpha}}dy,
\endaligned
\end{equation}
where $R_k=|x^k|$ and $x_{R_k}$ are selected such that
$$B_{R_k}(x_{R_k})\subset \Sigma_{\lambda_k}~~\text{and}~|x_{R_k}|=M_kR_k.$$
By the decay condition \eqref{20230921-e4}, we can choose $M_k$ such that
$$0<u(y)\leq\frac{C_1}{M_k^\beta R_k^\beta}\leq\frac{C_2}{R_k^\beta}\leq u(x^k)\leq\frac{C_0}{R_k^\beta}<\varepsilon,\quad \forall y\in B_{R_k}(x_{R_k}),$$
similar to \eqref{e33:20230921}.
Then by condition (G2) and \eqref{e40:20230921}, we have
 \begin{equation}\label{e41:20230921}
\aligned
&F_{G,K}u_{\lambda_k}(x^k)-F_{G,K}u(x^k)\\
\leq &C_{n,\alpha}c_0w_{\lambda_k}(x^k)\int_{B_{R_k}(x_{R_k})}\frac{(u(x^k)-u(y))^\gamma}{|x^k-y^{\lambda_k}|^{n+\alpha}}dy\\
\leq &C_{n,\alpha}C w_\lambda(x^k)\frac{1}{{R_k}^{\beta \gamma+\alpha}}.
\endaligned
\end{equation}
Combining \eqref{e39:20230921} with \eqref{e41:20230921}, we get
$$\frac{C}{{R_k}^{\beta \gamma+\alpha}}\leq  u^{s}(x^k).$$
This contradicts assumption \eqref{20230921-e4}. Hence the sequence $\{x^k\}$ must be bounded.

Now from \eqref{e38:20230921}, we have
$$w_{\lambda_o}(x^o)\leq 0,~~\text{hence}~ x^o\in \partial\Sigma_{\lambda_o};~~~~
\text{and}~~ \nabla w_{\lambda_o}(x^o) = 0.$$
It follows that
$$\frac{w_{\lambda_k}(x^k)}{\delta_k}\to0,\quad
\text{as}~ k\to+\infty.$$
Then by \eqref{e39:20230921}, we get
\begin{equation*}\label{e42:20230921}
\limsup\limits_{\delta_k\to0}\frac{1}{\delta_k}\left[F_{G,K}(u_{\lambda_k}(x^k))-F_{G,K}(u(x^k))\right]\geq 0.
\end{equation*}
This contradicts Theorem \ref{t3:20230921}. Hence \eqref{e36:20230921} holds. Since $x_1$ direction can be chosen arbitrarily, we conclude that $u$ is radially symmetric
about some point.
This completes the proof of Theorem \ref{th1.2}-(iii).
\end{proof}

\vskip4mm
{\section{The limit of $\mathcal{L}_Ku(x)$ as $\alpha\to2$. }}
 \setcounter{equation}{0}
In this section, we investigate the limit of $\mathcal{L}_Ku(x)$ as $\alpha\to2$ for each fixed $x$.

{\bf The proof of Theorem \ref{th1.3}-(i).}
\begin{proof}
First fix $\epsilon>0$, we divide $\mathcal{L}_{\mathcal{K}}u(x)$ into two parts:
\begin{equation}\nonumber
\aligned
\mathcal{L}_{\mathcal{K}}u(x)=&\frac{4n}{\omega_n}\int_{\R^n\setminus B_\epsilon(x)}\frac{e^{-|x-y|^2}}{\Gamma(\frac{2-\alpha}{2})}\frac{u(x)-u(y)}{|x-y|^{n+\alpha}}dy
+\frac{4n}{\omega_n}\int_{ B_\epsilon(x)}\frac{e^{-|x-y|^2}}{\Gamma(\frac{2-\alpha}{2})}\frac{u(x)-u(y)}{|x-y|^{n+\alpha}}dy\\
:=&I_1+I_2.
\endaligned
\end{equation}
By $u\in C^{1,1}_{loc}\cap L^\infty(\R^n)$, it is easy to verify that
\begin{equation}\label{2023112-e11}
\aligned
\lim\limits_{\alpha\to{2^-}}I_1=0.
\endaligned
\end{equation}
Let $z=y-x$, for $|x-y|\leq\epsilon$, by Taylor expansion
$$u(x)-u(y)=-\nabla u(x)\cdot z-\frac{1}{2}\partial_{ij}u(x)z_iz_j+ O(\epsilon)|x-y|^2,$$
we get
\begin{equation}\label{2023112-e115}
\aligned
I_2=&\frac{4n}{\omega_n\Gamma(\frac{2-\alpha}{2})}\int_{B_\epsilon(x)}\frac{(-\nabla u(x)\cdot z)e^{-|x-y|^2}}{|x-y|^{n+\alpha}}dy+\frac{4n}{\omega_n\Gamma(\frac{2-\alpha}{2})}\int_{B_\epsilon(x)}\frac{O(\epsilon)|x-y|^2e^{-|x-y|^2}}{|x-y|^{n+\alpha}}dy \\
&-\frac{2n}{\omega_n\Gamma(\frac{2-\alpha}{2})}\int_{ B_\epsilon(x)}\frac{\partial_{ij}u(x)z_iz_je^{-|x-y|^2}}{|x-y|^{n+\alpha}}dy\\
:=&II_1+II_2+II_3.
\endaligned
\end{equation}
Due to symmetry,
\begin{equation}\label{2023112-e12}
\aligned
II_1=\frac{4n}{\omega_n\Gamma(\frac{2-\alpha}{2})}\int_{ B_\epsilon(0)}\frac{(-\nabla u(x)\cdot z)e^{-|z|^2}}{|z|^{n+\alpha}}dz=0.
\endaligned
\end{equation}
Next, we show that
\begin{equation}\label{2023112-e13}
\lim\limits_{\alpha\to2^-}II_2=O(\epsilon).
\end{equation}
Indeed, let
\begin{equation}\nonumber
\aligned
II_2=2nO(\epsilon)\frac{2}{\omega_n\Gamma(\frac{2-\alpha}{2})}\int_{ B_\epsilon(0)}\frac{e^{-|z|^2}}{|z|^{n+\alpha-2}}dz:=2nO(\epsilon)III,
\endaligned
\end{equation}
where, \begin{equation}\nonumber
\aligned
III
=&\frac{1}{\Gamma(\frac{2-\alpha}{2})}\int_0^{\epsilon^2}e^{-t}t^{-\frac{\alpha}{2}}dt\\
\in &\frac{1}{\Gamma(\frac{2-\alpha}{2})}\left[\int_0^{\epsilon^2}t^{\frac{2-\alpha}{2}-1}e^{-\epsilon^2}dt, \int_0^{\epsilon^2}t^{\frac{2-\alpha}{2}-1}dt\right] \\
=&\left[\frac{e^{-\epsilon^2}\epsilon^{2-\alpha}}{\frac{2-\alpha}{2}\Gamma(\frac{2-\alpha}{2})},\frac{\epsilon^{2-\alpha}}{\frac{2-\alpha}{2}\Gamma(\frac{2-\alpha}{2})}\right].
\endaligned
\end{equation}
Then by $$\lim\limits_{\alpha\to2^-}\frac{2-\alpha}{2}\Gamma(\frac{2-\alpha}{2})=\lim\limits_{\alpha\to2^-}\Gamma(\frac{2-\alpha}{2}+1)=1,$$
we get \begin{equation}\label{20231018-e1}
\aligned
\lim\limits_{\alpha\to2^-}III\in \left[e^{-\epsilon^2},1\right],
\endaligned
\end{equation}
and thus \eqref{2023112-e13} holds.

For the remaining part $II_3$, we estimate as follows
\begin{equation}\nonumber
\aligned
II_3
=&-\frac{2n}{\omega_n\Gamma(\frac{2-\alpha}{2})}\partial_{ij}u(x)\int_{ B_\epsilon(0)}\frac{z_iz_je^{-|z|^2}}{|z|^{n+\alpha}}dz\\
=&-\frac{2n}{\omega_n\Gamma(\frac{2-\alpha}{2})}\partial_{ii}u(x)\int_{ B_\epsilon(0)}\frac{z_i^2e^{-|z|^2}}{|z|^{n+\alpha}}dz\\
=&-\Delta u(x)\cdot\frac{2}{\omega_n\Gamma(\frac{2-\alpha}{2})}\int_{ B_\epsilon(0)}\frac{e^{-|z|^2}}{|z|^{n+\alpha-2}}dz\\
=&-\Delta u(x)III.
\endaligned
\end{equation}
By \eqref{20231018-e1}, we get
\begin{equation}\label{2023112-e14}
\lim\limits_{\alpha\to2^-}II_3\in\left[-\Delta u(x)e^{-\epsilon^2},-\Delta u(x)\right].
\end{equation}
Finally, let $\epsilon\to0^+$, combining \eqref{2023112-e11}, \eqref{2023112-e12}, \eqref{2023112-e13} and \eqref{2023112-e14}, we complete
the proof.
\end{proof}

{\bf The proof of Theorem \ref{th1.3}-(ii).}
\begin{proof}
First fix $\epsilon>0$, we divide $\mathcal{L}_{\mathbb{K}}u(x)$ into two parts:
$$\mathcal{L}_{\mathbb{K}}u(x)=\int_{\R^n\setminus B_\epsilon(0)}\left(2-\alpha\right)\frac{u(x)-u(y)}{\|x-y\|^{n+\alpha}}dy+\int_{ B_\epsilon(0)}\left(2-\alpha\right)\frac{u(x)-u(y)}{\|x-y\|^{n+\alpha}}dy:=I_1+I_2.$$
By $u\in C^{1,1}_{loc}\cap L^\infty(\R^n)$, it is easy to verify that
\begin{equation}\label{2023112-e21}
\aligned
\lim\limits_{\alpha\to{2^-}}I_1=0.
\endaligned
\end{equation}
Similar to \eqref{2023112-e115}, by Taylor expansion we get
\begin{equation}\nonumber
\aligned
I_2=&\left(2-\alpha\right)\int_{B_\epsilon(x)}\frac{-\nabla u(x)\cdot z}{\|x-y\|^{n+\alpha}}dy+\left(2-\alpha\right)\int_{B_\epsilon(x)}\frac{O(\epsilon)|x-y|^2}{\|x-y\|^{n+\alpha}}dy \\
&-\left(2-\alpha\right)\int_{ B_\epsilon(x)}\frac{\partial_{ij}u(x)z_iz_j}{\|x-y\|^{n+\alpha}}dy\\
:=&II_1+II_2+II_3.
\endaligned
\end{equation}
Due to symmetry,
\begin{equation}\label{2023112-e22}
\aligned
II_1=\left(2-\alpha\right)\int_{ B_\epsilon(0)}\frac{(-\nabla u(x)\cdot z)}{\|z\|^{n+\alpha}}dz=0.
\endaligned
\end{equation}
By the equivalence of norms in $\R^n$,
$$c_{n,p}|y|\leq\|y\|\leq c'_{n,p}|y|,$$
we have
\begin{equation}\label{2023112-e23}
\aligned
\left|II_2\right|\leq (2-\alpha)\left|O(\epsilon)\right|c^{-n-\alpha}_{n,p}\int_{B_\epsilon(x)}\frac{1}{|x-y|^{n+\alpha-2}}dy=\left|O(\epsilon)\right|c^{-n-\alpha}_{n,p}\omega_n\epsilon^{2-\alpha}.
\endaligned
\end{equation}
%Given $r, l>0$ and $x\in\R^n$, we will denote the anisotropic ball
%$$E_{r,l}(x):=\left\{y\in\R^n: \sum_{i=1}^n\frac{|y_i-x_i|^2}{r^{4/p}}<l^2\right\}.$$
%The anisotropic scaling is defined by
%$$T_r e=r^{\frac{2}{p}} e.$$
For the remaining part $II_3$, we estimate as follows
\begin{equation}\nonumber
\aligned
II_3
=&-\left(2-\alpha\right)\partial_{ij}u(x)\int_{ B_{\epsilon}(0)}\frac{z_iz_j}{\|z\|^{n+\alpha}}dz\\
=&-\left(2-\alpha\right)\partial_{ii}u(x)\int_{ B_{\epsilon}(0)}\frac{z_i^2}{\|z\|^{n+\alpha}}dz\\
=&-\frac{1}{n}\left(2-\alpha\right)\Delta u(x)\int_{ B_{\epsilon}(0)}\frac{|z|^2}{\|z\|^{n+\alpha}}dz \\
=&-\frac{1}{n}\left(2-\alpha\right)\Delta u(x)\epsilon^{2-\alpha}\int_{ B_{1}(0)}\frac{| y|^2}{\|y\|^{n+\alpha}}dy.
\endaligned
\end{equation}
By the equivalence of norms in $\R^n$,
we obtain
$$\frac{{c_{n,p}}^{-n-\alpha}\omega_n }{2-\alpha}\leq\int_{B_1(0)}\frac{|y|^2}{\|y\|^{n+\alpha}}dy\leq \frac{{c'_{n,p}}^{-\alpha-n} \omega_n}{2-\alpha}.$$
Therefore,
\begin{equation}\label{2023112-e24}
\lim\limits_{\alpha\to2^-} II_3=-C_{n,p}\Delta u(x).
\end{equation}
Finally, let $\epsilon\to0^+$, combining \eqref{2023112-e21}, \eqref{2023112-e22}, \eqref{2023112-e23} and \eqref{2023112-e24}, we complete the proof.
\end{proof}

\vskip4mm
{\section{ Acknowledgements }}
\setcounter{equation}{0}
 The authors would like to thank Professor Congming Li for many precious suggestions.

\vskip4mm
{\section{ Data Availability }}
\setcounter{equation}{0}
 No data was used for the research described in the article.

\vskip4mm
{\section{ Conflict of Interest }}
\setcounter{equation}{0}
The authors declared that they have no conflict of interest.

\vskip4mm
%%%%%%%%%%%%%%%%%%%%


\begin{thebibliography}{10}

\bibitem{berestycki1988monotonicity}
{\sc H.~Berestycki and L.~Nirenberg}, {\em Monotonicity, symmetry and antisymmetry of
  solutions of semilinear elliptic equations}, Journal of Geometry and Physics 5(2) (1988), 237-275.

\bibitem{berestycki1991method}
{\sc H.~Berestycki and L.~Nirenberg}, {\em On the method of moving planes and the
  sliding method}, Boletim da Sociedade Brasileira de Matem{\'a}tica-Bulletin/Brazilian Mathematical Society 22(1) (1991),
  1-37.

\bibitem{BS} {\sc K. Bogdan and P. Sztonyk}, {\em Estimates of the potential kernel and Harnack's inequality for the anisotropic fractional Laplacian}, Studia Mathematica 181(2) (2007), 101-123.

\bibitem{BCPS} {\sc C. Br\"{a}ndle, E. Colorado, A. de Pablo and U. S\'{a}nchez}, {\em
A concave-convex elliptic problem involving the fractional Laplacian}, Proc. Roy. Soc. Edinburgh Sect. A 143(1) (2013), 39-71.



\bibitem{CXSY}
{\sc X. Cabr\'{e} and Y. Sire}, {\em Nonlinear equations for fractional Laplacians I: regularity, maximum principles, and
Hamiltonian estimates}, Ann. Inst. Henri Poincare Anal. Non Linaire 31(1) (2014), 23-53.

\bibitem{caffarelli}
{\sc L. Caffarelli, B.~Gidas, and J.~Spruck}, {\em Asymptotic symmetry and local
  behavior of semilinear elliptic equations with critical sobolev growth}, Comm. Pure Appl. Math. 42(3) (1989), 271-297.

\bibitem{CS}
{\sc L. Caffarelli and L. Silvestre}, {\em An Extension Problem Related to the Fractional Laplacian}, Comm. Partial Differential Equations 32 (2007), 1245-1260.

\bibitem{CS1} {\sc L. Caffarelli and L. Silvestre}, {\em Regularity theory for fully nonlinear integro-differential equations}, Comm. Pure Appl. Math. 62 (2009), 597-638.

\bibitem{CRS}  {\sc
L. Caffarelli, X. Ros-Oton and J. Serra}, {\em
Obstacle problems for integro-differential operators: regularity of solutions and free boundaries}, Invent. math. 208 (2017), 1155-1211.

\bibitem{CFY}
{\sc W. Chen, Y. Fang and R. Yang}, {\em Liouville theorems involving the fractional
Laplacian on a half space}, Adv. Math. 274 (2015), 167-198.

\bibitem{chen1997priori}
{\sc W. Chen and C. Li}, {\em A priori estimates for prescribing scalar curvature
  equations}, Annals of mathematics (1997), 547-564.

\bibitem{CL} {\sc W. Chen and C. Li}, {\em Maximum principles for the fractional p-Laplacian and symmetry of solutions}, Adv. Math. 335 (2018), 735-758.

\bibitem{chenw}
{\sc W. Chen and C. Li}, {\em Classification of solutions of some nonlinear
  elliptic equations}, Duke Mathematical Journal 63(3) (1991), 615-622.

\bibitem{CLLG} {\sc W. Chen, C. Li and G. Li}, {\em Maximum principles for a fully nonlinear fractional
order equation and symmetry of solutions}, Calc. Var. (2017) 56:29.

\bibitem{CLL} {\sc W. Chen, C. Li and Y. Li}, {\em A direct method of moving planes for the fractional Laplacian}, Adv. Math. 308 (2017), 404-437.

%\bibitem{chen2}
%W.~Chen, C.~Li, and Y.~Li, ``A direct method of moving planes for the
 % fractional laplacian,'' \emph{Advances in Mathematics}, vol. 308, pp.
 % 404--437, 2017.

%\bibitem{chen4}
%W.~Chen, C.~Li, and B.~Ou, ``Classification of solutions for an integral
 % equation,'' \emph{Communications on Pure and Applied Mathematics: A Journal
 % Issued by the Courant Institute of Mathematical Sciences}, vol.~59, no.~3,
 % pp. 330--343, 2006.

\bibitem{chen2017direct}
{\sc W. Chen, Y. Li, and R. Zhang}, {\em A direct method of moving spheres on fractional
  order equations}, Journal of Functional Analysis 272(10) (2017), 4131-4157.

\bibitem{CLO}
{\sc W. Chen, C. Li and B. Ou}, {\em Classification of solutions for an integral equation}, Comm. Pure Appl. Math. 59(3) (2006), 330-343.

\bibitem{CZSR}
{\sc Z.-Q. Chen and R. Song}, {\em Estimates on Green function and Poisson kernels for symmetric stable processes}, Math. Ann. 312 (1998), 465-501.

\bibitem{cheng2017maximum}
{\sc T. Cheng, G. Huang, and C. Li}, {\em The maximum principles for fractional
  laplacian equations and their applications}, Commun. Contemp. Math. 19(6) (2017), 1750018.

\bibitem{cheng2017direct}
{\sc C. Cheng, Z. L{\"u} and Y. L{\"u}}, {\em A direct method of moving planes for the
  system of the fractional laplacian}, Pacific Journal of Mathematics 290(2) (2017), 301-320.

\bibitem{FW}
{\sc P. Felmer and Y. Wang}, {\em Radial symmetry of positive solutions involving the fractional Laplacian}, Commun. Contemp. Math. 16 (2014), 1350023.

\bibitem{FLS}
{\sc R. L. Frank, E. Lenzmann and L. Silvestre}, {\em Uniqueness of radial solutions for the fractional Laplacian}, Comm. Pure Appl. Math. 69(9) (2016), 1671-1726.

\bibitem{JW}
{\sc S. Jarohs and T. Weth}, {\em
Symmetry via antisymmetric maximum principles in nonlocal problems of variable order}, Ann. Mat. Pura Appl. 195 (2016), 273-291.

\bibitem{SYK} {\sc S. Kim, Y-C. Kim and K-A. Lee}, {\em
Regularity for Fully Nonlinear Integro-differential Operators with Regularly Varying Kernels}, Potential Anal. 44 (2016), 673-705.

\bibitem{li}
{\sc C. Li}, {\em Local asymptotic symmetry of singular solutions to nonlinear elliptic
  equations}, Inventiones mathematicae 123(1) (1996), 221-231.

\bibitem{LiZhuo} {\sc D. Li and R. Zhuo}, {\em An integral equation on half space}, Proc. Am. Math. Soc. 8 (2010), 2779-2791.

\bibitem{OS} {\sc X. Ros-Oton and J. Serra}, {\em The Dirichlet problem for the fractional Laplacian: regularity up to the boundary}, Journal de Math\'{e}matiques Pures et Appliqu\'{e}es 101(3) (2014), 275-302.

\bibitem{PS}
{\sc P. Sztonyk}, {\em Regularity of harmonic functions for anisotropic fractional Laplacians}, Math. Nachr. 283 (2010), 289-311.

%徐美清新增文献


\end{thebibliography}
\end{document}